\documentclass[a4paper,twopage,reqno,11pt]{amsart}
\usepackage[top=30mm,right=30mm,bottom=30mm,left=30mm]{geometry}

\usepackage{tabularx}
\usepackage{graphicx}
\usepackage{amsmath}
\usepackage{amssymb}
\usepackage{amsfonts}
\usepackage{amsthm}
\usepackage{amstext}
\usepackage{amsbsy}
\usepackage{amsopn}
\usepackage{amscd}
\usepackage{enumerate}
\usepackage{color}
\usepackage[colorlinks]{hyperref}
\usepackage[hyperpageref]{backref}
\usepackage{multirow}
\usepackage{rotating}
\usepackage{longtable}
\usepackage{float}
\usepackage{lscape}
\usepackage{pdflscape}
\usepackage{url}

\newtheorem{theorem}{Theorem}[section]
\newtheorem{corollary}[theorem]{Corollary}
\newtheorem{lemma}[theorem]{Lemma}

\theoremstyle{definition}

\numberwithin{equation}{section}

\newcommand{\SL}{\mathrm{SL}}

\newcommand{\PSL}{\mathrm{PSL}}

\newcommand{\AGL}{\mathrm{AGL}}
\newcommand{\PGL}{\mathrm{PGL}}
\newcommand{\PGaL}{\mathrm{P\Gamma L}}

\newcommand{\A}{\mathrm{A}}

\renewcommand{\S}{\mathrm{S}}

\newcommand{\D}{\mathrm{D}}

\newcommand{\Aut}{\mathrm{Aut}}

\newcommand{\PG}{\mathrm{PG}}

\newcommand{\M}{\mathrm{M}}

\newcommand{\Dmc}{\mathcal{D}}
\newcommand{\Bmc}{\mathcal{B}}

\newcommand{\Pmc}{\mathcal{P}}

\newcommand{\Omc}{\mathcal{O}}

\renewcommand{\leq}{\leqslant}
\renewcommand{\geq}{\geqslant}

\newcommand{\imod}[1]{\allowbreak\mkern4mu({\operator@font mod}\,\,#1)}

\begin{document}
\title[Alternating groups and $2$-designs]{Alternating groups as flag-transitive automorphism groups of $2$-designs with block size seven}

\author[A. Daneshkhah]{Ashraf Daneshkhah}
\address{Ashraf Daneshkhah, Department of Mathematics, Faculty of Science, Bu-Ali Sina University, Hamedan, Iran.
}
\email{adanesh@basu.ac.ir}
\email{daneshkhah.ashraf@gmail.com}

\date{\today}%

\begin{abstract}
In this article, 
we show that if $\Dmc$ is a $2$-design with block size $7$ admitting flag-transitive almost simple automorphism group with socle an alternating group, then $\Dmc$ is $\PG_{2}(3,2)$ with parameter set $(15,7,3)$ and $G=\A_7$, or $\Dmc$ is the $2$-design with parameter set $(55, 7, 1680)$ and $G=\A_{11}$ or $\S_{11}$. 
\end{abstract}

\maketitle

\section{Introduction}\label{sec:intro}

The $2$-$(v,k,\lambda)$ designs with highly symmetries have been of most interest during last decades, in particular, flag-transitive $2$-designs. There have been numerous contribution to classify flag-transitive $2$-designs with $\lambda=1$ which are also known as Steiner $2$-designs or linear spaces, and in conclusion, a classification of such incidence structures has been announced in 1990 \cite{a:BDDKLS90}. There are several interesting results with restriction on the certain parameters of $2$-designs, see for example \cite{a:ABFGMRTZ-CP,a:ABCD-PrimeRep,a:ABD-PrimeLam,a:ACD-PrimeRep-class,a:ADM-PrimeLam-An,a:Zhou-SqrRep-An,a:Zhou-PrimLam-An}.

In this paper, we are interested in studying flag-transitive $2$-designs with small block size $k$. If $k$ is small, then we have several well-known examples of flag-transitive $2$-designs. Steiner triple designs are $2$-$(v,3,1)$ designs which have been extensively studied, see \cite[ch II.2]{b:Handbook}.  For $k=4$, Zhan, Zhou and Chen \cite{a:smallk-4-PA} proved that a flag-transitive automorphism group of a $2$-$(v,4,\lambda)$ design  is point-primitive of affine, almost simple or product type, and they obtained all such possible designs with product type automorphism groups. The almost simple case for $k=4$ or $5$ when the socle is respectively $\PSL_{2}(q)$ or a sporadic simple group have been treated, see \cite{a:smallk-5-spor,a:smallk-4-PSL2}. The $2$-designs with block size $6$  admitting flag-transitive and point-imprimitive automorphism groups have been determined in \cite{a:smallk-6-Imp}. We note that if $G$ is a $2$-homogeneous automorphism group on the point-set $\Pmc$ and $B$ is a $k$-subset of $\Pmc$ with $k\geq 2$, then $\Dmc=(\Pmc,B^{G})$ is a $2$-design, and if moreover, $B$ is an orbit of a subgroup of $G$, then $G$ is flag-transitive on $\Dmc$. 
Therefore, we mainly focus on the case where $G$ is not $2$-homogeneous. It follows immediately from \cite[Theorem 1.4]{a:Zhou-kprime} that a flag-transitive automorphism group $G$ of a $2$-design with $k$ prime must be point-primitive, and it is of affine, or almost simple type.
We in particular focus on the case where $G$ is an almost simple group with socle an alternating group and obtain all possible flag-transitive $2$-designs with $k=7$: 

\begin{theorem}\label{thm:alt}
    Let $\Dmc$ be a nontrivial $2$-$(v,7,\lambda)$ design, and let $G$ be a flag-transitive automorphism group of $\Dmc$. If $G$ is point-primitive of almost simple type with socle an alternating group $\A_{c}$ with $c\geq 5$, then one of the following holds:
    \begin{enumerate}[\rm (a)]
    	\item $\Dmc$ is $\PG_{2}(3,2)$ with parameter set $(15,7,3)$ and $G=\A_7$  with point-stabiliser $\PSL_3(2)$;
    	\item $\Dmc$ is a $2$-design with parameter set $(55, 7, 1680)$ and $G=\A_{11}$ or $\S_{11}$ with point-stabiliser $\S_{9}$ or $2\times \S_{9}$, respectively. 
    \end{enumerate}
\end{theorem}

In order to prove Theorem~\ref{thm:alt}, for the case where $v< 100$, by \cite{a:Tang-v100}, we obtain the $2$-designs in the statement. Then we assume that $v\geq 100$, and in this case, we show that there is no $2$-$(v,7,\lambda)$ deign admitting flag-transitive and point-primitive automorphism group $G$. Here, we first observe that the point-stabiliser $H$ of $G$ has to be large, that is to say, $|G|\leq |H|^{3}$. The possibilities for $H$ can be read off from \cite{a:AB-Large-15}. In Section \ref{sec:thm}, we examine these possibilities and prove our desired result.

\subsection{Definitions and notation}\label{sec:defn}
All groups and incidence structures in this paper are finite. A group $G$ is said to be \emph{almost simple} with socle $X$ if $X\unlhd G\leq \Aut(X)$, where $X$ is a nonabelian simple group. Symmetric and alternating groups on $c$ letters are denoted by $\S_{c}$ and $\A_{c}$, respectively. We write ``$n$'' for group of order $n$. 
A $2$-$(v,k,\lambda)$ design $\Dmc$ is a pair $(\Pmc,\Bmc)$ with a set $\Pmc$ of $v$ points and a set $\Bmc$ of $b$ blocks such that each block is a $k$-subset of $\Pmc$ and each pair of distinct points is contained in exactly $\lambda$ blocks. We say that $\Dmc$ is nontrivial if $2 < k < v-1$. A \emph{flag} of $\Dmc$ is a point-block pair $(\alpha, B)$ such that $\alpha \in B$. An \emph{automorphism} of $\Dmc$ is a permutation on $\Pmc$ which maps blocks to blocks and preserving the incidence. The \emph{full automorphism} group $\Aut(\Dmc)$ of $\Dmc$ is the group consisting of all automorphisms of $\Dmc$.  For $G\leq \Aut(\Dmc)$, $G$ is called \emph{flag-transitive} if $G$ acts transitively on the set of flags. The group $G$ is said to be \emph{point-primitive} if $G$ acts primitively on $\Pmc$. 
For a given positive integer $n$ and a prime divisor $p$ of $n$, we denote the $p$-part of $n$ by $n_{p}$, that is to say, $n_{p}=p^{t}$ with $p^{t}\mid n$ but $p^{t+1}\nmid n$. Further notation and definitions in both design theory and group theory are standard and can be found, for example in~\cite{b:Beth-I,b:Dixon}.

\section{Preliminaries}\label{sec:pre}

In this section, we state some useful facts in both design theory and group theory. 

\begin{lemma}\label{lem:param}
    Let $\Dmc$ be a $2$-design with parameter set $(v,k,\lambda)$. Then
    \begin{enumerate}[\rm (a)]
        \item $r(k-1)=\lambda(v-1)$;
        \item $vr=bk$;
        \item $v\leq b$ and $k\leq r$;
        \item $\lambda v<r^2$.
    \end{enumerate}
\end{lemma}
\begin{proof}
    Parts (a) and (b) follow immediately by simple counting. The inequality $v\leq b$ is the Fisher's inequality \cite[p. 57]{b:Dembowski}, and so by applying part (b), we have that $k\leq r$. By part (a) and (c), we easily observe that  $r^{2}>r(k-1)=\lambda(v-1)>\lambda v$, and so $\lambda v<r^{2}$, as desired.
\end{proof}

If a group $G$ acts transitively on a set $\Pmc$ and $\alpha\in \Pmc$, the \emph{subdegrees} of $G$ are the length of orbits of the action of the point-stabiliser $G_\alpha$ on $\Pmc$.

\begin{lemma}\label{lem:six}
    Let $\Dmc$ be a $2$-design with parameter set $(v,k,\lambda)$, and let $\alpha$ be a point of $\Dmc$. If $G$ a flag-transitive automorphism group of $\Dmc$, then
    \begin{enumerate}[\rm (a)]
        \item $r\mid |G_{\alpha}|$;
        \item $r\mid \lambda d$, for all nontrivial subdegrees $d$ of $G$.
    \end{enumerate}
\end{lemma}
\begin{proof}
    Since $G$ is flag-transitive, the point-stabiliser $G_{\alpha}$ is transitive on the  set of all blocks containing $\alpha$, and so $r=|G_{\alpha}:G_{\alpha,B}|$. Thus $r$ divides $|G_{\alpha}|$.  Part (b) is proved in \cite[p. 9]{a:Davies-87}.
\end{proof}

\begin{corollary}\label{cor:subdeg}
    Let $\Dmc=(\Pmc,\Bmc)$ be a   $2$-$(v,k,\lambda)$ design with $\alpha\in \Pmc$ admitting a flag-transitive automorphism group $G$. Then $v-1$ divides $\gcd(k-1,\lambda(v-1)) d$, for all nontrivial subdegrees $d$ of $G$. Moreover, if $H=G_{\alpha}$, then 
    \begin{align}\label{eq:order}
        |G|\leq \gcd(k-1,\lambda(v-1))|H|^{2}+|H|.
    \end{align}
\end{corollary}
\begin{proof}
    We know by Lemma~\ref{lem:param} that $\lambda(v-1)=r(k-1)$. Then $\lambda(v-1)/\gcd(k-1,\lambda(v-1))$ divides $r$. Since $\Dmc$ is flag-transitive, Lemma~\ref{lem:six}(b) implies that $r$ divides $\lambda d$ for all nontrivial subdegrees $d$ of $G$. Thus $\lambda(v-1)/\gcd(k-1,\lambda(v-1))$ divides $\lambda d$, and hence $v-1$ divides $\gcd(k-1,\lambda(v-1)) d$, for all nontrivial subdegrees $d$ of $G$. Moreover, since $d\leq |H|$ and $v=|G|/|H|$, the inequality \eqref{eq:order} holds.   
\end{proof}

\section{Proof of Theorem~\ref{thm:alt}.}\label{sec:thm}

Suppose that $\Dmc=(\Pmc,\Bmc)$ is a $2$-$(v,k,\lambda)$ design admitting flag-transitive and point-primitive automorphism group $G$ with socle $X$ an alternating group $\A_{c}$ of degree $c\geq 5$ on $\Omega=\{1,\ldots,c\}$ and that $H:=G_{\alpha}$ with $\alpha\in \Pmc$. Then $H$ is maximal in $G$  by \cite[Corollary 1.5A]{b:Dixon}, and since $G=HX$, we conclude that
\begin{align}
    v=\frac{|X|}{|H\cap X|}.\label{eq:v}
\end{align}

If $v< 100$, then by \cite[Theorem 1.2]{a:Tang-v100,a:Braic-255}, we obtain two $2$-designs: one is a  $2$-design with parameters $(55, 7, 1680)$ with  $G=\A_{11}, \S_{11}$, and the other one is $\PG_{2}(3,2)$ with parameters $(15,7,3)$, and $G=\A_7$  with the point-stabiliser $H=\PSL_3(2)$. We observe by \cite{a:Regueiro-alt-spor,a:Zhou-lam2-nonsym-An,a:Zhuo-CP-sym-alt,a:Zhou-CP-nonsym-alt} that there is no example of $2$-design with $\lambda=2$ or $\gcd(r,\lambda)=1$ admitting a flag-transitive alternation automorphism group.
Therefore, we assume that $v\geq 100$, $\lambda\geq 3$ and $\gcd(r,\lambda)\neq 1$. This in particular shows that $r=\lambda(v-1)/(k-1)\geq 99/2$, and hence $r\geq 49$. If $|H|\leq 6$, then by \eqref{eq:order}, we observe that $|G|\leq 6^3+6=222$ implying that $G=\A_{5}$ or $\S_{5}$, then by \cite[p.2]{b:Atlas}, we conclude that $v\leq 10$ which has already been considered. Therefore, we can assume that $|H|\geq 7$, and hence \eqref{eq:order} implies that $|G|\leq |H|^{3}$. 
Let $H_{0}:=H\cap X$. Then by \cite[Theorem 2 and Proposition 6.1]{a:AB-Large-15}, one of the following holds:
\begin{enumerate}[\rm \quad (i)]
    \item   $H_{0}$ is intransitive on $\Omega=\{1,\ldots,c\}$;
    \item   $H_{0}$ is transitive and imprimitive on $\Omega=\{1,\ldots,c\}$;
    \item   $G=\S_c$ and $(c,H)$ is one of the following:
    \begin{align*}
        \begin{array}{llll}
            \left(5,\AGL_1(5)\right), &  \left(6,\PGL_2(5)\right), &  \left(7,\AGL_1(7)\right), &  \left(8,\PGL_2(7)\right), \\
            \left(9,\AGL_2(3)\right), & \left(10,\A_6{\cdot}2^{2}\right), &  \left(12,\PGL_2(11)\right);
        \end{array}
    \end{align*}
    \item  $G=\A_6{\cdot}2=\PGL_2(9)$ and $H$ is $\D_{20}$ or a Sylow $2$-subgroup $P$ of $G$ of order $16$;
    \item  $G=\A_6{\cdot}2=\M_{10}$ and $H$ is $\AGL_1(5)$ or a Sylow $2$-subgroup $P$ of $G$ of order $16$;
    \item  $G=\A_6{\cdot}2^{2}=\PGaL_2(9)$ and $H$ is $\AGL_1(5){\times} 2$ or a Sylow $2$-subgroup $P$ of $G$ of order $32$;
    \item  $G=\A_{c}$ and $(c,H)$ is one of the following:
    \begin{align*}
        \begin{array}{llll}
            (5,\D_{10}), & (6,\PSL_2(5)), & (7,\PSL_2(7)), & (8,\AGL_3(2)), \\
            (9,3^{2}{\cdot}\SL_2(3)), & (9,\PGaL_2(8)), & (10,\M_{10}), & (11,\M_{11}), \\
            (12,\M_{12}), & (13,\PSL_3(3)), & (15,\A_8), & (16,\AGL_4(2)), \\
            (24,\M_{24}).
        \end{array}
    \end{align*}
\end{enumerate} 

For the cases (iii)-(vii),  it is easy to see that the only possibilities $(G,H)$ with $v\geq 100$ satisfying \eqref{eq:order}  are $(\S_{8},\PGL_{2}(7))$ and $(\A_{9},\PGaL_{2}(8))$ for $v=120$, and  $(\A_{11},\M_{11})$ and $(\A_{12},\M_{12})$ for $v=2520$. These cases cannot occur as for each of these possibilities, the parameter $b$ is a divisor of $|G|$, and for each such $b$, and for $v\leq b$, we cannot find any parameters $r$ and $\lambda$ satisfying Lemma~\ref{lem:param}(a). Therefore, $H_{0}$ is either intransitive, or imprimitive.\smallskip 

\noindent \textbf{(i)} Suppose that $H_{0}=(\S_s \times \S_{c-s})\cap \A_{c}$ is intransitive on $\Omega=\{1,\ldots,c\}$ with $1\leq s< c/2$. In this case, $H=(\S_s \times \S_{c-s})\cap G$. Note that $H$ is maximal in $G$ as long as $s\neq c-s$. 
Note also that $H_{0}$ contains all the even permutations of $H$, and hence
$H_{0}=H$ if $G=\A_{c}$, or the index of $H_{0}$ in $H$ is $2$ if $G=\S_c$. Since $G$ is flag-transitive, $H$ is transitive on the set of blocks  passing through $\alpha$. Hence $H$ fixes exactly one point in $\Pmc$, and so it stabilises exactly one $s$-subset, say $\Delta$, in $\Omega$. Therefore, we can identify the point $\alpha$ of $\Pmc$ with the unique $s$-subset $\Delta$ of $\Omega$ stabilised by $H$. Thus $v=\binom{c}{s}$. Since $H_{0}$ acting on $\Omega$ is intransitive, it has at least two orbits. According to \cite[p. 82]{a:Delandsheer-Lin-An-2001}, two points of $\Pmc$ are in the same orbit under $H_{0}$ if and only if the corresponding $s$-subsets $\Delta_{1}$ and $\Delta_{2}$ of $\Omega$ intersect $S$ in the same number of points. Thus $G$ acting on $\Pmc$ has rank $s+1$, and each  $H_{0}$-orbit $\Omc_{i}$ on  $\Pmc$ corresponds to a possible size $i\in \{0,1,\ldots, s\}$ and these are precisely the families of $s$-subsets of $\Omega$ that intersect $S$, see also \cite[Proposition 2.5]{a:Alavi-intrans}. Then if $d_i$ is the length of a $G$-orbit on $\Pmc$, then $d_0=1$, and $d_j=\binom{s}{j-1} \binom{c-s}{s-j+1}$ when $G=\A_{c}$ or $d_j= \binom{s}{j-1} \binom{c-s}{s-j+1}/2$ when $G=\S_c$ for $j=1,\ldots,s$. 

By Corollary~\ref{cor:subdeg}, we have that $v-1$ divides $\gcd(6,\lambda(v-1)) d_j$ for all nontrivial subdegrees $d_j$ of $G$. In particular, if we take $j=s$, then $v-1$ divides $\gcd(6,\lambda(v-1))\cdot s(c-s)$, and so $v-1\leq\gcd(6,\lambda(v-1))\cdot s(c-s)$, and hence
\begin{align*}
    v=\binom{c}{s}\leq 6s(c-s)+1.
\end{align*}
Set $t:=c-s$. Thus
\begin{align}\label{eq:int-s}
    \binom{s+t}{s}\leq 6st+1.
\end{align}

Let $s\geq 4$. If $t\geq 10$, then since $t>s$, we observe that $(t+1)^4/24>6t^2+1>6st+1$, and so $\binom{s+t}{s}\geq \binom{t+4}{4}=(t+1)(t+2)(t+3)(t+4)/24\geq(t+1)^4/24>6st+1$, which violates \eqref{eq:int-s}. Moreover, the inequality \eqref{eq:int-s} does not hold for $9\geq t>s\geq 4$. Therefore, $s=1,2,3$.  

If $s=1$, then $v=c\geq 5$. Note that $G$ is $(v-2)$-transitive on $\Pmc$. Since $2<k\leq v-2$, $G$ acts $k$-transitively on $\Pmc$. Then $b=|\Bmc|=|B^G|=\binom{c}{k}=\binom{v}{k}$, that is to say, $\Dmc$ is complete. 

If $s=2$, then $v=c(c-1)/2$ and by Corollary~\ref{cor:subdeg} and as noted above, $v-1$ divides $\gcd(6,\lambda(v-1))\cdot 2(c-2)$, and so there exists a positive integer $m$ such that $m[c(c-1)-2]=24(c-2)$. Thus $c(c-2)< mc(c-2)<m[c(c-1)-2]=24(c-2)$, and  hence $c\leq 23$. For these values of $c$, since $v-1$ divides $\gcd(6,\lambda(v-1))\cdot 2(c-2)$, we have that $v-1$ divides $12(c-2)$, and considering the fact that $v\geq 100$, we conclude that $(c,v)$ is one of the pairs 
$(10,109)$,
$(11,121)$,
$(12,133)$,
$(13,145)$,
$(14,157)$,
$(15,169)$,
$(16,181)$,
$(17,193)$,
$(18,103)$,
$(18,205)$,
$(19,109)$,
$(19,217)$,
$(20,115)$,
$(20,229)$,
$(21,121)$,
$(21,241)$,
$(22,127)$,
$(22,253)$,
$(23,133)$,
$(23,265)$. But none of these possibilities satisfies $v=c(c-1)/2$, which is a contradiction. 

If $s=3$, then $v=c(c-1)(c-2)/6$, and so Corollary~\ref{cor:subdeg} implies that $v-1$ divides $\gcd(6,\lambda(v-1))\cdot 3(c-3)$, and so $c(c-1)(c-2)-6<6\cdot18(c-3)$, and since $v\geq 100$, it follows that $(c,v)=(10,120)$, but then $v-1=119$ does not divide $18(c-3)=126$, which is a contradiction.\smallskip

\noindent \textbf{(ii)} Suppose now that $H_{0}$ is transitive and imprimitive on $\Omega=\{1,\ldots,c\}$. In this case, $H=(\S_s \wr \S_{c/s})\cap G$ is imprimitive, where $s$ divides $c$ and $2\leq s\leq c/2$. 
Indeed, $H_{0}$ is transitive and imprimitive on $\Omega=\{1,\ldots,c\}$, $H_{0}$ acting on $\Omega$ preserves a partition $\Sigma$ of $\Omega$ into $t$ classes of size $s$ with $t\geq 2$, $s\geq 2$ and $c=st$. Thus $H_{0} \leq G_\Sigma<G$. Since $G$ is isomorphic to $\S_c$ or $\A_{c}$ and since both natural actions of $G$ and $X$ on $\Omega$ are primitive, we conclude that $H_{0}$ contains all the even permutations of $\Omega$ preserving the partition $\Sigma$. By the same argument as in \cite[Case 2]{a:Delandsheer-Lin-An-2001}, \cite[(3.2)]{a:Regueiro-alt-spor} and \cite[p. 1489-1490]{a:Zhuo-CP-sym-alt}, the imprimitive partition  $\Sigma$ is the only nontrivial partition of $\Omega$ preserved by $H_{0}$. Since $X$ acts transitively on all the partitions of $\Omega$ into $t$ classes of size $s$, we can identify the points of $\Dmc$ with the partitions of $\Omega$ into $t$ classes of size $s$, and so $v=\binom{ts}{s}\binom{(t-1)s}{s} \cdots\binom{3s}{s}\binom{2s}{s}/(t!)$, that is to say,
\begin{align}\label{eq:alt-imp-v}
    v=\binom{ts-1}{s-1}\binom{(t-1)s-1}{s-1}\cdots \binom{3s-1}{s-1}\binom{2s-1}{s-1}.
\end{align}

We note that the suborbits of $G$ on $\Omega$ can be described by the notion of $j$-cyclics introduced in \cite[p. 84]{a:Delandsheer-Lin-An-2001}. Indeed, if  a partition $\Sigma_{1}$ of $\Omega$ is a point of $\Pmc$, then 
for $j=2,\ldots,t$, the set $\Gamma_j$ of $j$-cyclic partitions with respect to $\Sigma_{1}$ is a union of $H$-orbits on $\Pmc$, see \cite[Case 2]{a:Delandsheer-Lin-An-2001} and \cite[p. 1490-1491]{a:Zhuo-CP-sym-alt}. Therefore, by Corollary~\ref{cor:subdeg}, we have that  $v-1$ divides $\gcd(6,\lambda(v-1))\cdot d_{s}$, where
\begin{align}\label{eq:subdegs-An}
    d_{s} = \left.
    \begin{cases}
        s^{2}\binom{t}{2}, & \text{ if } s\geq 3; \\
        t(t-1), & \text{ if } s=2.
    \end{cases}
    \right.	
\end{align}

If $s=2$, then $t\geq3$ as $c=st\geq5$. By \eqref{eq:alt-imp-v},  we have that $v=\prod_{i=0}^{t-2}[2t-(2i+1)]$ and since $v-1$ divides $\gcd(6,\lambda(v-1)) d_{2}=\gcd(6,\lambda(v-1))\cdot t(t-1)$, it follows that
\begin{align*}
    \prod_{i=0}^{t-2}[2t-(2i+1)]-1\leq 6 t(t-1),
\end{align*}
which is true when $t=2,3$, and so $v=3,15$, respectively, which is a contradiction. 

If $s\geq 3$, then since, 
\begin{align*}
    \binom{is-1}{s-1}=\frac{is-1}{s-1}\cdot \frac{is-2}{s-2}\cdots \frac{is-(s-1)}{1}>i^{s-1}
\end{align*} 
with $2\leq i\leq t$, by \eqref{eq:alt-imp-v}, we conclude that $v>t^{(s-1)(t-1)}$. Since also $v-1$ divides $\gcd(6,\lambda(v-1)) d_s=\gcd(6,\lambda(v-1)) s^{2}\binom{t}{2}$, we deduce by Corollary~\ref{cor:subdeg} that
\begin{align*}
    t^{(s-1)(t-1)}\leq 3 s^2t(t-1).
\end{align*}
Thus $t^{(s-1)(t-1)-2}<3s^2$, which is true when $(s,t)=(3,3)$ or $t=2$ and $s=3,\ldots,11$. For each such pair $(s,t)$, the fact that $v-1$ divides $6 s^{2}\binom{t}{2}$ implies that $(s,t)=(3,2)$ in which $v=10$, which is a contradiction. This completes the proof.

\section*{Acknowledgements}

The author is grateful to Alice Devillers and Cheryl E. Praeger for supporting her visit to UWA (The University of Western Australia) during February–June 2023. She also thanks Bu-Ali Sina University for the support during her sabbatical leave. 





\end{document}